\documentclass[11pt,reqno]{amsart}

\newtheorem{theorem}{Theorem}[section]
\newtheorem{lemma}[theorem]{Lemma}
\newtheorem{proposition}[theorem]{Proposition}

\theoremstyle{definition}

\theoremstyle{remark}
\newtheorem{remark}[theorem]{Remark}

\numberwithin{equation}{section}

\begin{document}

\title{Using Twisted Alexander Polynomials to Detect Fiberedness}

\author{Azadeh Rafizadeh}
\address{Department of Physics and Mathematics, William Jewell College}
\curraddr{Liberty, MO 64068-1896}
\email{rafizadeha@william.jewell.edu}

\begin{abstract}
In this paper we use twisted Alexander polynomials to prove
that the exterior of a particular graph knot is not fibered.
Then we build three 2-component graph links out of this knot,
and use similar techniques to discuss their fiberedness.
\end{abstract}

\maketitle

\textbf{Acknowledgement:} This work was done in preparation for the author's dissertation. I would like to thank my advisor, Stefano Vidussi.

\section{Introduction and Main Results}

The main purpose of this paper is to find explicit applications
of the relationship between twisted Alexander polynomials and
fiberedness. For twisted Alexander polynomials, we will follow the definition given in \cite{FK}. In particular, we studied a graph knot $K$
that is included in a homology sphere $\Sigma$ (different from
the 3-dimensional sphere $S^3$), and three 2-component links
that have $K$ as one of their components. For graph links, D.
Eisenbud and W. Neumann introduced splice diagrams and
developed a method to use the combinatorial information
included in splice diagrams to determine fiberedness and the
Thurston norm, \cite{EN}. However, the technique they use only
applies to graph links, whereas the method of this paper can
theoretically be applied to any 3-manifold. \\
The following theorem of C. McMullen shows the ability of the
(ordinary) Alexander polynomial to provide information on the
Thurston norm and fiberedness for a general 3-manifold $N$. If
$\phi=(m_1,...,m_n)\in H^1(N;\mathbb{Z})$, then div$(\phi)$ is the greatest common
divisor of $m_1, ..., m_n$.

\begin{theorem}\label{mcm}
  \textnormal{(McMullen, \cite{McMullen}) } Let $N$ be a compact, connected,
 orientable 3-manifold whose boundary (if any) is a union of
 tori. Then for any
 $\phi \in H^1(N;\mathbb{Z})$ \

$$deg(\Delta_{N,\phi})\leq \|\phi\|_T+\left\{\begin{array}{cc}
            0,& b_1(N)\geq 2\\
            div (\phi)\cdot (1+b_3(N)),& b_1(N)=1 \\
\end{array}\right.$$

Moreover, if $\phi$ is fibered, $\Delta_{N,\phi}$ is monic, and
equality holds.
\end{theorem}

It is well-known that the converse of Theorem \ref{mcm} is not true as we will
show for the graph knot $K$, which has the splice diagram shown in Figure 1.\newpage

\begin{proposition}\label{calcsfork} The genus of the knot $K$
is $1$, it has Alexander polynomial equal to $t^2-t+1$, and it is
not fibered.\end{proposition}

 \vskip-.45in
\setlength{\unitlength}{0.4cm}
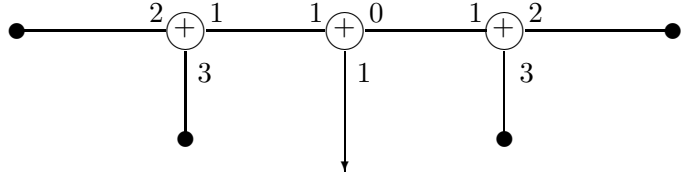
\begin{figure}[h!]
 $$\begin{picture}(23,3)
 \put(0,0){\line(1,0){4.95}}
 \put(0,0){\circle*{0.5}}
 \put(5.6,0){\circle{1.2}}
 \put(5.6,-0.7){\line(0,-1){3}}
 \put(5.6,-3.6){\circle*{0.5}}
 \put(6.3,0){\line(1,0){3.9}}
 \put(10.9,0){\circle{1.2}}
 \put(10.9,-0.7){\vector(0,-1){4}}
 \put(11.6,0){\line(1,0){4}}
 \put(16.2,0){\circle{1.2}}
 \put(16.2,-0.7){\line(0,-1){3}}
 \put(16.2,-3.6){\circle*{0.5}}
 \put(16.9,0){\line(1,0){5}}
 \put(21.8,0){\circle*{0.5}}
 \put(5.2,-0.15){$+$}
 \put(4.4,0.3){$2$}
 \put(6.4,0.3){$1$}
 \put(6,-1.7){$3$}
 \put(10.5,-0.15){$+$}
 \put(9.7,0.3){$1$}
 \put(11.7,0.3){$0$}
 \put(11.3,-1.7){$1$}
 \put(15.8,-0.15){$+$}
 \put(15,0.3){$1$}
 \put(17,0.3){$2$}
 \put(16.7,-1.7){$3$}
\end{picture}$$\vskip.5in
\caption{Splice Diagram of the Knot $K$}
\label{knotk}
\end{figure}

S. Friedl and T. Kim have generalized the result in Theorem
\ref{mcm} by considering the collection of twisted Alexander
polynomials in the following theorem.

\begin{theorem} \label{FK}

\textnormal{(Friedl-Kim, \cite{FK})} Let $N$ be a 3-manifold
different from $S^1 \times D^2$ and $S^1 \times S^2$. Let $\phi
\in H^1(N; \mathbb{Z})$ be such that $(N,\phi)$
fibers over $S^1$. Then for every representation $\alpha:
\pi_1(N) \rightarrow Gl_k(\mathbb{Z})$,
\begin{equation}\label{eqn} \Delta_{N,\phi}^\alpha \textnormal{ is monic
and }deg(\Delta_{N,\phi}^\alpha) = k\|\phi\|_T+ deg(\Delta
_{N,\phi,0})+ deg(\Delta_{N,\phi,2}). \end{equation}

\noindent Here, $\Delta _{N,\phi,0}$ and $\Delta _{N,\phi,2}$ are
determined by the Alexander modules $H_0(N;\mathbb{Z}^k[F])$
and $H_2(N;\mathbb{Z}^k[F])$.

\end{theorem}

Theorem \ref{FK} leads one to believe that the collection of
twisted Alexander polynomials gives stronger obstructions to
fiberedness. This, in fact is confirmed by the following
theorem of S. Friedl and S. Vidussi.

\begin{theorem}\label{FV}

\textnormal{(Friedl-Vidussi, \cite{FV}) } Let $N$ be a compact,
connected, orientable 3-manifold whose boundary (if any) is a
union of tori. Let $\phi$ be non-trivial in
$H^1(N;\mathbb{Z})$.
 Then if $\phi$ is not fibered, there is a
representation $\alpha: \pi_1(N) \to Gl_k(\mathbb{Z})$ for
which the conditions in \eqref{eqn} are not satisfied.

\end{theorem}

For knots of genus 1, this result has been enhanced to show that,
there is some representation $\alpha$ for which the twisted
Alexander polynomial vanishes, \cite{V}. This result has been further generalized to any 3-manifold pair, $(N,\phi)$, where $\phi\in H^1(N; \mathbb{Z})$, \cite{FV2}.

The proof of Theorem \ref{FV} is not constructive. We have found
explicit representations for the knot $K$ and one 2-component
link containing $K$, for which \eqref{eqn} is violated.

\begin{theorem} For the representation $\alpha:\pi_1(K) \rightarrow S_5
\rightarrow GL_5(\mathbb{Z})$ given in Theorem \ref{mainthm}, 
$\Delta_{K,\phi}^\alpha$ is not monic.
\end{theorem}

In order to find the explicit representation, we will first
calculate the fundamental group of the exterior of $K$, and
then use the computer program Knottwister written by S.
Friedl,\cite{KT}.\\

\section{Proof of Proposition \ref{calcsfork}}

To prove the proposition, we use various results from \cite{EN}. (More details can be found in \cite{diss}.)\\

\begin{proof} As we can see in the diagram in Figure 2, there is
one arrowhead vertex, we will call this vertex $v_1$.
Considering the conventions in \cite{EN}, this knot has $8$
vertices. So $n=1$, and $k=8$. First, we will find $l_{ij}$ for
$i=1$, and $1<j\leq8$ : $l_{12}= l_{13}= l_{14}=l_{15}=0,
l_{16}=6, l_{17}=3, l_{18}=2.$ \\

\setlength{\unitlength}{0.35cm}
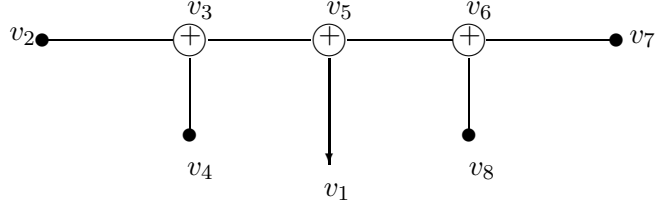
\begin{figure}
 $$\begin{picture}(23,3)
 \put(0,0){\line(1,0){4.95}}
 \put(0,0){\circle*{0.5}}
 \put(-1.25,0){$v_2$}
 \put(5.6,0){\circle{1.2}}
 \put(5.5,1){$v_3$}
 \put(5.6,-0.7){\line(0,-1){3}}
 \put(5.6,-3.6){\circle*{0.5}}
 \put(5.5,-5.1){$v_4$}
 \put(6.3,0){\line(1,0){3.9}}
 \put(10.9,0){\circle{1.2}}
 \put(10.8,1){$v_5$}
 \put(10.9,-0.7){\vector(0,-1){4}}
 \put(10.7,-5.9){$v_1$}
 \put(11.6,0){\line(1,0){4}}
 \put(16.2,0){\circle{1.2}}
 \put(16.1,1){$v_6$}
 \put(16.2,-0.7){\line(0,-1){3}}
 \put(16.2,-3.6){\circle*{0.5}}
 \put(16.2,-5.1){$v_8$}
 \put(16.9,0){\line(1,0){5}}
 \put(21.8,0){\circle*{0.5}}
 \put(22.3,-0.1){$v_7$}
 \put(5.2,-0.11){$+$}

 \put(10.5,-0.11){$+$}
 \put(15.8,-0.11){$+$}

\end{picture}$$\vskip.65in
\caption{Vertices of the Knot $K$} \label{kwithvertices}
\end{figure}

 For
boundary vertices and arrowhead vertices, $\delta_i=1$. For
this particular knot, each node has $3$ arrowhead vertices
and/or boundary vertices attached to it. So we have the following values for $\delta_i$ where $1<i\leq 8$:\\
$\delta_2=\delta_4=\delta_7=\delta_8 =1$ and
$\delta_3=\delta_5=\delta_6 =3$. Now we use Theorem 12.1 in \cite{EN} to compute the Alexander polynomial:
$$\Delta=(t-1)(t^0-1)^{-1}(t^0-1)(t^0-1)^{-1}(t^0-1)(t^6-1)(t^3-1)^{-1}(t^2-1)^{-1}.$$

\noindent Following the convention mentioned in \cite{EN} we cancel the terms $(t^0-1)$ and
$(t^0-1)^{-1}$. Doing so we get
$$\Delta=\frac{(t-1)(t^6-1)}{(t^3-1)(t^2-1)}=\frac{t^3+1}{t+1}=t^2-t+1.$$

To find the genus of the knot, we calculate the Thurston norm
of the class $\phi=(1)\in H^1(\Sigma
\setminus(\nu(K),\mathbb{Z})\simeq \mathbb{Z}$. By Theorem 11.1 in
\cite{EN},

$$\|\phi\|_T = \|(1)\|_T=\sum_{j=2}^8(\delta_j-2)|l_{1j}|=1.$$

So this knot has genus equal to $1$ as claimed since
$\|\phi\|_T=2g-1$. It remains to show it is not fibered. To
show this, we use Theorem 11.2 in \cite{EN}, which assets that if some of
the terms in the summation are zero, as in our case, then $K$ is not fibered.  \end{proof}

\section{Proof of the Main Theorem}
\subsection{The Fundamental Group}

To find the explicit representation $\alpha$, we first need to
calculate the fundamental group of its exterior. For a knot in $S^3$, one can
use the Wirtinger presentation of any blackboard projection of
the knot to compute its fundamental group.  Given that the knot $K$ is contained in a homology sphere
$\Sigma$, this method is not directly available, because we do not
have access to any blackboard presentation. The route we will
follow uses instead the Seifert-Van Kampen theorem and the
decomposition of the knot exterior into three components
reflected by the splice diagram of $K$ given in Figure 1.

From now on, for the sake of simplicity, when we talk about the
fundamental group of the exterior of a link or a knot $L$, we
will call it the fundamental group of $L$. We will follow
this convention in our notation as well. For example, we will
denote the fundamental group of the exterior of the knot $K$ as
$\pi_1(K)$ instead of $\pi_1(\Sigma \setminus (\nu(K)))$.

\begin{lemma}\label{pi1k} The exterior of the knot $K$ has the following fundamental
group:

$\pi_1(K)=\langle x,y,s,t,b |xyx=yxy,
stbst=bstb,\\
 \indent xs=sx, xt=tx,
 s=x^{-1}yx^2yx^{-3}, x=(st)^{-1}b(st)^2b(st)^{-3}\rangle.$\\
 \end{lemma} \vskip-.2in

\begin{proof} First, we will look at the three building blocks of the
splice diagram. If we separate the middle node from the rest,
we get the following splice diagram.\vskip-.2in

\setlength{\unitlength}{0.4cm}

\begin{figure}[h!]
 $$\begin{picture}(23,3)
\put(10.9,0){\circle{1.2}}
 \put(10.9,-0.7){\vector(0,-1){2.5}}
 \put(11.6,0){\vector(1,0){3}}
\put(10.2,0){\vector(-1,0){3}}
 \put(11.7,0.3){$0$}
 \put(11.3,-1.7){$1$}
 \put(10.5,-.15){$+$}
 \put(9.8,0.3){$1$}
\end{picture}$$\vskip.25in
\caption{Splice Diagram of the 3-Component Necklace}
\end{figure}
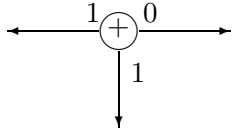

The three-component necklace that this splice diagram
represents is the one in Figure \ref{3compnl}. The arrowhead
vertex with weight $0$ is the main loop, and the ones with
weight $1$ are the two hanging loops. We will call the main
loop $N_0$, the loop hanging on the left $N_1$ and the one
hanging on the right $N_2$. The following is its projection.

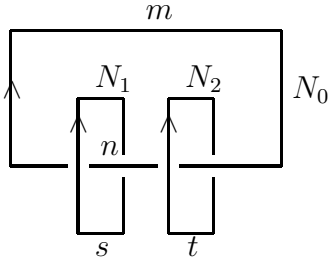
\begin{figure}[h!]
\setlength{\unitlength}{0.15cm}
$$\begin{picture}(23,1.5)
\thicklines
 \put(0,0){\line(1,0){24}}
 \put(12,0){\line(-2,-1){.75}}
 \put(12,0){\line(-2,1){.75}}
 \put(0,0){\line(0,-1){12}}
 \put(24,0){\line(0,-1){12}}
 \put(0,-12){\line(1,0){5}}
 \put(7,-12){\line(1,0){6}}
 \put(15,-12){\line(1,0){9}}
 \put(11,-12){\line(2,-1){.75}}
 \put(11,-12){\line(2,1){.75}}
 \put(6,-18){\line(0,1){12}}
 \put(6,-6){\line(1,0){4}}
 \put(6,-18){\line(1,0){4}}
 \put(10,-18){\line(0,1){5}}
 \put(10,-6){\line(0,-1){5}}
 \put(6,-8){\line(-1,-2){.4}}
 \put(6,-8){\line(1,-2){.4}}
 \put(14,-18){\line(0,1){12}}
 \put(14,-6){\line(1,0){4}}
 \put(14,-18){\line(1,0){4}}
 \put(18,-18){\line(0,1){5}}
 \put(18,-6){\line(0,-1){5}}
 \put(14,-8){\line(-1,-2){.4}}
 \put(14,-8){\line(1,-2){.4}}
 \put(-.7,-6){$\wedge$}
 \put(5.2,-9){$\wedge$}
 \put(13.2,-9){$\wedge$}
 \put(12,1){\large$m$}
 \put(15.7,-20){\large$t$}
 \put(7.5,-20){\large$s$}
 \put(8,-11){\large$n$}
 \put(25,-6){\large$N_0$}
 \put(15.5,-5){\large$N_2$}
 \put(7.5,-5){\large$N_1$}
\end{picture}$$ \vskip1in
\caption{$3$-Component Necklace}\label{3compnl}
\end{figure}

 To avoid making the
diagrams busy, we put the names of the meridians on the arc and
will not include the actual meridians in pictures.
 For this necklace, let $\mu(N_1)=s$, and $\mu(N_2)=t$ be the
meridians of $N_1$ and $N_2$. Also since $N_0$ is made of two
arcs $m$ and $n$, we can choose as meridian of this component
either $m$ or $n$. Using the Wirtinger presentation for links,
we see that the (simplified) fundamental group of this link is
$\pi_1(N)=\langle n,s,t| ns=sn, nt=tn\rangle.$\\

The node on the left is the $(2,3)$ cable on the unknot, as we can read from its splice diagram (Proposition 7.3 in \cite{EN}). Hence it represents the right-handed trefoil knot
with the canonical orientation.
 We will call it $T_L$. The
diagram in Figure 5 shows the node on the left separated from the rest.

\setlength{\unitlength}{0.35cm}
\hskip-1in
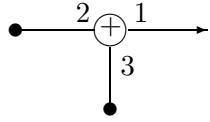
\begin{figure}
 $$\begin{picture}(11,3)
 \put(0.5,0){\line(1,0){3}}
 \put(0.5,0){\circle*{0.5}}
 \put(4.1,0){\circle{1.2}}
 \put(4.1,-0.65){\line(0,-1){2.1}}
 \put(4.1,-3){\circle*{0.5}}
 \put(4.8,0){\vector(1,0){3}}
\put(3.7,-0.16){$+$}
 \put(2.8,0.3){$2$}
 \put(5,0.3){$1$}
 \put(4.5,-1.7){$3$}
\end{picture}$$\vskip.2in
\caption{Splice Diagram of the Trefoil on the Left}
\end{figure}

Considering the projection of the right-handed trefoil shown in
Figure 6, we can use the Wirtinger presentation for knots to
calculate the fundamental group. Doing so will give us the
following (simplified) fundamental group: $\pi_1(T_L)= \langle
x,y | xyx=yxy\rangle.$ For this knot, we will choose the
meridian to be $\mu(T_L)=x$. Then by the details discussed in
Remark 3.13 of \cite{B},
 the longitude will be $\lambda(T_L)= zxyx^{-3}=x^{-1}yx^2yx^{-3}$.

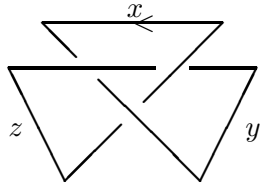
\begin{figure}[h!]

\setlength{\unitlength}{0.15cm}
$$\begin{picture}(23,1.5)
\thicklines
 \put(4,0){\line(1,0){16}}
 \put(4,0){\line(1,-1){3}}
 \put(9,-5){\line(1,-1){9}}
 \put(20,0){\line(-1,-1){7}}
 \put(11,-9){\line(-1,-1){5}}
 \put(6,-14){\line(-1,2){5}}
 \put(18,-14){\line(1,2){5}}
 \put(1,-4){\line(1,0){13}}
 \put(17,-4){\line(1,0){6}}
 \put(12,-.7){$<$}
 \put(11.5,.5){$x$}
 \put(22,-10){$y$}
 \put(1,-10){$z$}
\end{picture}$$

 \vskip.75in
\caption{The Trefoil Knot on the Left, $T_L$}\label{trefoilxyz}

\end{figure}

 Splicing on the left, we identify the longitude of $T_L$ with the meridian of $N_1$ and the meridian of $T_L$
with the longitude of $N_1$. Doing so will yield the
relations $s=x^{-1}yx^2yx^{-3}$ and $x=n$ respectively.\\
Since the node on the right is another copy of the right-handed
trefoil knot, we will call it $T_R$. This knot has the
fundamental group $\pi_1(T_R)=\langle a,b |aba=bab\rangle.$ If
we choose its meridian to be $a$, then the longitude is
$caba^{-3}=a^{-1}ba^2ba^{-3}$. The splicing on the right
happens along the $N_0$ component of the necklace, with
meridian $\mu(N_0)=n$ and longitude $\lambda(N_0)=st$. Hence
after splicing on the right, we will have the relations $st=a$
and $n=a^{-1}ba^2ba^{-3}$.\\
Given the fundamental groups of each of the building blocks,
along with the relations due to the splicing, the Seifert-Van
Kampen Theorem states that the fundamental group of the knot
$K$ is: $\pi_1(K)=\langle x, y, n, s, t, a, b | xyx=yxy,
aba=bab, ns=sn, nt=tn, x=n, s=x^{-1}yx^2yx^{-3}, st=a,
n=a^{-1}ba^2ba^{-3}\rangle.$ Simplifying this group, we get:
$\pi_1(K)= \langle x, y, s, t, b| xyx=yxy, stbst=bstb,
sx=xs,xt=tx,
s=x^{-1}yx^2yx^{-3},x=(st)^{-1}b(st)^2b(st)^{-3}\rangle$. \end{proof}

\subsection{Finding an Explicit Representation $\alpha$ that Shows $K$ is not Fibered}
 In this section, using the above presentation of $\pi_1(K)$ we find an explicit representation of $\pi_1(K) \rightarrow
 GL_5(\mathbb{Z})$ for which the twisted Alexander polynomial is not monic. To do so, we use the computer program
 Knottwister.

\begin{theorem}\label{mainthm}
For the representation $\alpha:\pi_1(K) \rightarrow S_5
\rightarrow GL_5(\mathbb{Z})$ given by
$$\alpha(a)=(15234), \alpha(b)=(13524), \alpha(n)=(14523), \alpha(s)=(12345)$$
$$\alpha(t)=(15234), \alpha(x)=(14523), \alpha(y)=(34125),$$
$\Delta_{K,\phi}^\alpha$ is not monic. (Here, one-line permutation notation is used.)
\end{theorem}

\begin{proof} Knottwister takes the fundamental group of
$K$ along with a cohomology class $\phi$ as the input data. For
knots, $\phi$ can be chosen to be the abelianization map $\phi:
\pi_1(K) \rightarrow \mathbb{Z}$. To identify explicitly the
abelianization map $\phi$ we add the commutator relations to
the fundamental group found in Lemma \ref{pi1k}. Then the map $\phi$
is given explicitly as:
$$\phi(x)=\phi(y)=\phi(s)=0 \textnormal{ and } \phi(b)=\phi(t)=1.$$

  It can be easily checked that $\alpha$ is a homomorphism, meaning that
 it respects the relations of the fundamental group. The ordinary Alexander polynomial is $t^2-t+1$, which is
identical to that of the trefoil knot. However, Knottwister
gives the twisted Alexander polynomial $\Delta_{K,
\phi}^\alpha$ with coefficients modulo $p$ for different prime
numbers. The twisted Alexander polynomial given by this
particular representation $\alpha$ over $\Bbb{F}_{5}[t^{\pm1}],
\Bbb{F}_{7}[t^{\pm 1}], \Bbb{F}_{11}[t^{\pm}],
\Bbb{F}_{13}[t^{\pm 1}],\\
  \Bbb{F}_{17}[t^{\pm 1}],
  \Bbb{F}_{19}[t^{\pm 1}],
  \Bbb{F}_{23}[t^{\pm 1}],$ and
  $\Bbb{F}_{29}[t^{\pm 1}]$ is equal to 0. Since the twisted Alexander polynomial associated with any one of these
representations vanishes, it is not monic.

\end{proof}

We can conclude from the previous theorem and Theorem \ref{FK}
that the knot $K$ is not fibered. Clearly, having the
polynomial vanish over any of the fields above would be
sufficient to show it is not monic. However, the fact that it
vanishes over all these fields is a strong evidence that it is
indeed $0$. Since the
genus of $K$ is 1 as we saw in Proposition \ref{calcsfork},
this observation is consistent with the enhanced
version of Theorem \ref{FV} appearing in \cite{V}.

\section{2-Component Links Containing $K$}

In this section, we discuss three 2-component links that contain the knot $K$ as a component. These links are the result of adding an arrowhead vertex to the three nodes of the splice diagram of $K$.\\

\subsection{The Link $L_\alpha$}
First, we put the second arrowhead vertex on
the last node. The following is the splice diagram of this
2-component link. From now on, we call this link $L_\alpha$.
Since this link contains the knot $K$ as a component, we can
denote it as $L_\alpha = K_\alpha \bigcup K$, when $K_\alpha$
is the new component of the link.\vskip.7in

\setlength{\unitlength}{0.3cm} 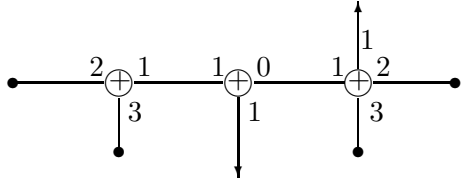
\begin{figure}[h!]
\vskip-.7in
 $$\begin{picture}(23,3)
 \put(0.9,0){\line(1,0){4}}
 \put(0.9,0){\circle*{0.5}}
 \put(5.6,0){\circle{1.2}}
 \put(5.6,-0.7){\line(0,-1){2.5}}
 \put(5.6,-3.1){\circle*{0.5}}
 \put(6.3,0){\line(1,0){3.9}}
 \put(10.9,0){\circle{1.2}}
 \put(10.9,-0.7){\vector(0,-1){3.5}}
 \put(11.6,0){\line(1,0){4}}
 \put(16.2,0){\circle{1.2}}
 \put(16.2,-0.7){\line(0,-1){2.5}}
 \put(16.2,-3.1){\circle*{0.5}}
 \put(16.9,0){\line(1,0){3.7}}
 \put(20.5,0){\circle*{0.5}}
 \put(5.15,-0.23){$+$}
 \put(4.3,0.3){$2$}
 \put(6.4,0.3){$1$}
 \put(6,-1.7){$3$}
 \put(10.4,-0.23){$+$}
 \put(9.7,0.3){$1$}
 \put(11.7,0.3){$0$}
 \put(11.3,-1.7){$1$}
 \put(15.75,-0.23){$+$}
 \put(15,0.3){$1$}
 \put(17,0.3){$2$}
 \put(16.7,-1.7){$3$}
 \put(16.2,.6){\vector(0,1){3}}
 \put(16.3,1.5){$1$}
\end{picture}$$ \vskip.3in
\caption{Splice Diagram of the Link $L_\alpha$}\label{lalpha}
\end{figure}

Using the theorems in \cite{EN}, we can easily prove the following proposition. The proof is similar to that of \ref{calcsfork} and hence is omitted.

\begin{proposition}\label{calcsforlalpha}
The $2$-component link $L_\alpha$ in Figure \ref{lalpha} has
the following properties:\

 1. Its multivariable Alexander polynomial is:
  $$\Delta_{L_\alpha}(t_1,t_2)=(t_1^{12}-t_1^6+1)(t_1^4t_2^4+t_1^2t_2^2+1)(t_1^3t_2^3+1).$$

  2. For a general $\phi=(p,q)$, the Thurston norm is: $$\|\phi\|_T = 7|p+q|+12|p|.$$

  3.  If $N$ is the exterior of the link, the pairs $(N,(0,1))$
  and $(N,(1,-1))$ are not fibered.\end{proposition}

\begin{remark}\label{oapforlalpha}
From Proposition \ref{calcsforlalpha}, we can observe that for the
class $\phi=(1,-1)$ the single variable Alexander polynomial is
$$\Delta_{L_\alpha,\phi}=6(t-1)(t^{12}-t^6+1).$$
Even though $deg(\Delta_{L_\alpha,\phi})=\|\phi\|_T+1$, the
polynomial is not monic. So Theorem \ref{mcm} states that this
class is not fibered. However, for $\phi=(0,1)$, we have the following ordinary Alexander polynomial:
$$\Delta_{L_\alpha,\phi}= (t-1) (t^4+t^2+1) (t^3+1)=(t^6-1)(t^2-t+1).$$
In this case, the
Alexander polynomial is monic, and
$deg(\Delta_{L_\alpha,\phi})=8$. According to Theorem
\ref{mcm}, this result is compatible with fiberedness, but we
showed in Proposition \ref{calcsforlalpha} that it is not
fibered.
\end{remark}

\subsection{Fundamental Group of the Exterior of $L_\alpha$} In order to use twisted Alexander polynomials to discuss the fiberedness of $L_\alpha$, we need to calculate the fundamental group of its exterior.

\begin{lemma}\label{pi1lalpha}
The fundamental group of the exterior of $L_\alpha$ is:
$$\pi_1(L_\alpha)= \langle c, d, e, f,  g, h, i, j, k, l, o,
p, q, r, u, v, w, a, x, y, n, s, t | $$
$$ xyx=yxy, ns=sn, nt=tn, s=x^{-1}yx^2yx^{-3}, e=st, $$
$$gd=cg, ve=dv, cf=ec, pg=fp, vh=gv, wi=hw, aj=ia, ek=je,
rc=kr,$$
$$eo=le, rp=or, gq=pg, vr=qv, cu=rc, pv=up, hw=vh, ia=wi,
jl=aj \rangle.$$
\end{lemma}

\begin{proof} Again, we need to decompose the link over its three nodes. For the node on the left and the one in the middle, the calculations are identical to those of the knot $K$. For $L_\alpha$, the node on the right before splicing is shown in Figure \ref{sdofD}.

\setlength{\unitlength}{0.3cm}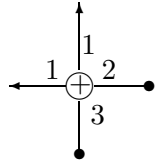
\begin{figure}[h!]
 $$\begin{picture}(40,3)
 \put(16.2,0){\circle{1.2}}
 \put(16.2,-0.7){\line(0,-1){2.5}}
 \put(16.2,-3){\circle*{0.5}}
 \put(16.9,0){\line(1,0){2.5}}
 \put(19.3,0){\circle*{0.5}}
\put(15.6,0){\vector(-1,0){2.5}}
 \put(15.75,-0.3){$+$}
 \put(14.7,0.3){$1$}
 \put(17.2,0.3){$2$}
 \put(16.7,-1.7){$3$}
 \put(16.2,.7){\vector(0,1){3}}
 \put(16.3,1.4){$1$}
\end{picture}$$\vskip.2in
\caption{Splice Diagram of the Link $D$ on the
Right}\label{sdofD}
\end{figure}
The splice diagram in Figure \ref{sdofD} represents a
2-component link, as it has two arrowhead vertices. It is the
$(2,3)$ cable on the right-handed trefoil (see Proposition 7.3
in \cite{EN}). Hence each component is a copy of the
right-handed trefoil knot, such that they have linking number
$6$. We call this 2-component
link $D$. The blackboard projection of the link $D$ is shown in Figure \ref{DT}. We only need to discuss the splicing relations on the right, as the ones on the left are identical to those of $K$. As for $K$, splicing on the
right happens along the main loop of the necklace, $N_0$. If we
choose to splice along the outer trefoil of $D$, and
choose its meridian to be $\mu(D)= e$, the longitude will be
$\lambda(D)= cpvwxergve^{-3}$. Hence the splicing relations
are:
$$n=cpvwxergve^{-3}, \textnormal{and } e=st.$$

\noindent Therefore, considering the fundamental groups of the
three building blocks of $L_\alpha$ and the relations that
result from splicing, we see that the fundamental group of the
exterior of $L_\alpha$ is:
$$\pi_1(L_\alpha)= \langle c, d, e, f,  g, h, i, j, k, l, o,
p, q, r, u, v, w, a, x, y, n, s, t | $$
$$ xyx=yxy, ns=sn, nt=tn, s=x^{-1}yx^2yx^{-3}, e=st, $$
$$gd=cg, ve=dv, cf=ec, pg=fp, vh=gv, wi=hw, aj=ia, ek=je,
rc=kr,$$
$$eo=le, rp=or, gq=pg, vr=qv, cu=rc, pv=up, hw=vh, ia=wi,
jl=aj \rangle.$$
\end{proof} \vskip-.5in

\begin{figure}[h!]

\setlength{\unitlength}{0.3cm}
$$\begin{picture}(23,1.5)
\thicklines
 \put(0,-6){\line(2,-3){2}}
 \put(0,-6){\line(-2,3){3}}
 \put(-3,-1.5){\line(1,0){22}}
 \put(3,-10.5){\line(2,-3){.6}}
 \put(4.5,-12.75){\line(2,-3){14.2}}
 \put(16,-6){\line(-2,-3){6.4}}
 \put(9,-16.5){\line(-2,-3){.6}}
 \put(16,-6){\line(2,3){3}}
 \put(7.5,-18.75){\line(-2,-3){10.5}}
 \put(-3,-10){\line(0,-1){24.5}}
 \put(-3,-10){\line(1,0){13}}
 \put(11.5,-10){\line(1,0){1}}
 \put(14,-10){\line(1,0){10}}
 \put(24,-10){\line(0,-1){5.5}}
 \put(24,-16.5){\line(0,-1){6}}
 \put(24,-23.5){\line(0,-1){7}}
 \put(24,-31.5){\line(0,-1){2.5}}
 \put(24,-34){\line(-1,0){5.4}}
 \put(2,-5){\line(2,-3){3}}
 \put(5.75,-10.63){\line(2,-3){.6}}
 \put(7,-12.5){\line(2,-3){12.4}}
 \put(2,-5){\line(1,0){12}}
 \put(14,-5){\line(-2,-3){5.7}}
 \put(7.5,-14.75){\line(-2,-3){.6}}
 \put(6,-17){\line(-2,-3){7}}
 \put(-1,-27.5){\line(0,1){15.5}}
 \put(-1,-12){\line(1,0){9.5}}
 \put(10,-12){\line(1,0){1}}
 \put(12.5,-12){\line(1,0){11}}
 \put(24.6,-12){\line(1,0){3.4}}
 \put(28,-12){\line(0,-1){4}}
 \put(28,-16){\line(-1,0){8}}
 \put(20,-16){\line(0,-1){4}}
 \put(20,-20){\line(1,0){3.4}}
 \put(24.6,-20){\line(1,0){3.4}}
 \put(28,-20){\line(0,-1){3}}
 \put(28,-23){\line(-1,0){8}}
 \put(20,-23){\line(0,-1){4}}
 \put(20,-27){\line(1,0){3.4}}
 \put(24.6,-27){\line(1,0){3.4}}
 \put(28,-27){\line(0,-1){4}}
 \put(28,-31){\line(-1,0){8.6}}
 \put(19,-3){\large$e$}
 \put(-4,-20){\large$c$}
 \put(9,-18){\large$d$}
 \put(2.1,-11.2){\large$f$}
 \put(14.8,-30){\large$g$}
 \put(24.5,-25.5){\large$h$}
 \put(24.5,-18.5){\large$i$}
 \put(19,-9.5){\large$j$}
 \put(12,-9.5){\large$k$}
 \put(17,-13.3){\large$l$}
 \put(10,-13){\large$o$}
 \put(-0.5,-16){\large$p$}
 \put(6.5,-14.5){\large$q$}
 \put(8,-6){\large$r$}
 \put(6.5,-11.2){\large$u$}
 \put(14,-22){\large$v$}
 \put(19,-24){\large$w$}
 \put(19,-18){\large$x$}
 \put(-3.5,-23){$\bigvee$}
 \put(-1.5,-23){$\bigvee$}
\end{picture}$$
\vskip4in \hskip-2in\caption{The Link $D$}\label{DT}
\end{figure}
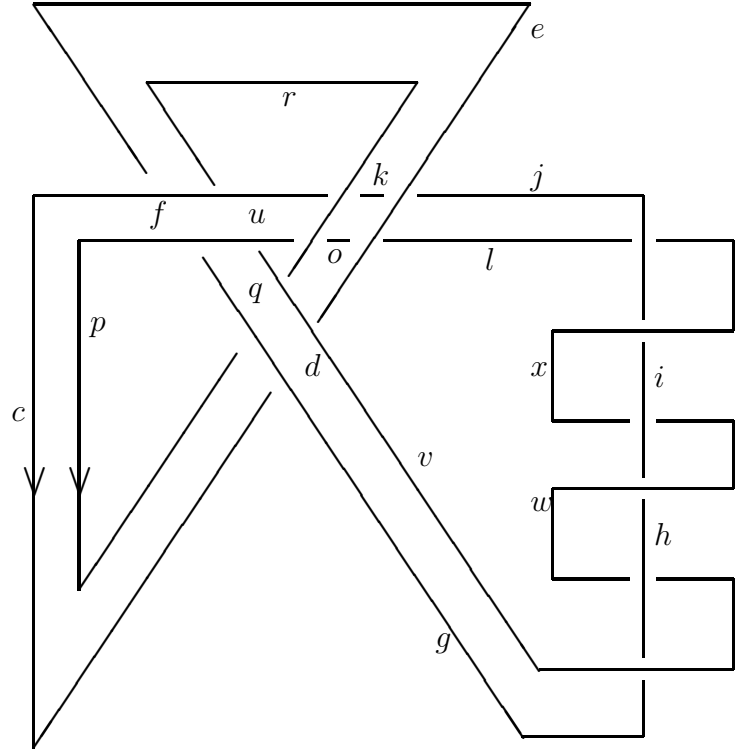

\subsection{Finding Representations for $\pi_1(L_\alpha)$ in Two Cases}
Since for all knots, the abelianization of their fundamental
group is isomorphic to $\mathbb{Z}$, if one cohomology class is
fibered, all are. However, it is possible that for the same link, some cohomology classes are fibered and others are
not. Now we will show that two
different cohomology classes for $L_\alpha$ are not fibered.

In the following theorem, we will find explicit
representations for which $\Delta_{N,\phi}^\alpha$ is not
monic, when $N$ is the exterior of $L_\alpha$ and $\phi$ is one
of the classes $(0,1)$ or $(1,-1)$. Consequently by Theorem
\ref{FK}, the pair $(N,\phi)$ is not fibered for either $\phi$.

\begin{theorem}\label{firsttwist} Let $N$ be the exterior of $L_\alpha$. For $\phi_1=(0,1)$, and $\phi_2=(1,-1)$, there are corresponding representations
$\alpha_1, \alpha_2:\pi_1(N) \rightarrow S_5 \to \mbox{GL}(\Bbb{Z},5)$ such
that $\Delta_{N,\phi_1}^{\alpha_1}$ and $\Delta_{N,\phi_2}^{\alpha_2}$ are not monic.
\end{theorem}

\begin{proof} First, we need to understand what $\phi_1$ does as a map.
We add all the commutator relations to the fundamental group in
Lemma \ref{pi1lalpha}. This will result in the following
relations:
$$c=d=e=f=g=h=i=j=k=t$$
$$o=l=p=q=r=u=v=w=a$$
$$s=1, x=y=n=v^6.$$

As expected for a 2-component link, the abelianization of
$\pi_1(L_\alpha)$ is isomorphic to $\mathbb{Z}\oplus
\mathbb{Z}$. We can see from the splice diagram of this link
that the two components that survive are one of the hanging
loops of the necklace, $N_2$, and the trefoil knot inside the
link $D$. These are the  arrowhead vertices in the splice
diagram. Hence $\phi_1$ is the homomorphism that sends $v$ to 0,
and $t$ to 1. Again, Knottwister takes the fundamental
group of $L_\alpha$ from Lemma \ref{pi1lalpha}, along with the
homomorphism $\phi_1$ as an input. In multiplicative notation,
$\phi_1$ is the following map:
$$\phi_1(c)=\phi_1(d)=\phi_1(e)=\phi_1(f)=\phi_1(g)=\phi_1(h)=\phi_1(i)=\phi_1(j)=\phi_1(k)=\phi_1(t)=1$$
$$\phi_1(a)=\phi_1(l)=\phi_1(o)=\phi_1(p)=\phi_1(q)=\phi_1(r)=\phi_1(u)=\phi_1(v)=\phi_1(w)=\phi_1(w)=\phi_1(a)$$
$$=\phi_1(s)=\phi_1(x)=\phi_1(y)=\phi_1(n)=0.$$

Knottwister gives the following representation
$\alpha_1:\pi_1(N) \to S_5 \to \mbox{GL}(\Bbb{Z},5)$, when the
elements in $S_5$ are written in on-line permutation form:
\[ \begin{array}{rclrclrclrcl}
a &\mapsto& (13245)&c &\mapsto& (23415)&d &\mapsto& (45321)&e &\mapsto& (24351)\\
f &\mapsto& (32514)&g &\mapsto& (13524)&h &\mapsto& (14532)&i &\mapsto& (15234)\\
j &\mapsto& (13524)&k &\mapsto& (31425)&l &\mapsto& (14325)&n &\mapsto& (45312)\\
o &\mapsto& (21345)&p &\mapsto& (21345)&q &\mapsto& (42315)&r &\mapsto& (21345)\\
s &\mapsto& (12345)&t &\mapsto& (24351)&u &\mapsto& (42315)&v &\mapsto& (14325)\\
w &\mapsto& (15342)&x &\mapsto& (45312)&y &\mapsto& (42513).\\
\end{array} \]

\noindent For this twist, the twisted Alexander polynomial, $\Delta_{N,\phi_1}^{\alpha_1}$, vanishes over the fields $\Bbb{F}_{7}[t^{\pm 1}], \Bbb{F}_{11,}[t^{\pm 1}],$
$ \Bbb{F}_{13}[t^{\pm 1}]$
$\Bbb{F}_{17}[t^{\pm 1}], \Bbb{F}_{19}[t^{\pm 1}], \Bbb{F}_{23}[t^{\pm 1}],$ and $\Bbb{F}_{29}[t^{\pm 1}]$. Since the twisted Alexander polynomial vanishes over
these finite fields, it cannot be monic. \\

Now, we do the same for $\phi_2=(1,-1)$. Using multiplicative notation, $\phi_2$ can be viewed as the map
that acts as follows on the generators of $\pi_1(L_\alpha)$:
$$\phi_2(c)=\phi_2(d)=\phi_2(e)=\phi_2(f)=\phi_2(g)=\phi_2(h)=\phi_2(i)=\phi_2(j)=\phi_2(k)=\phi_2(t)=-1$$
$$\phi_2(a)=\phi_2(l)=\phi_2(o)=\phi_2(p)=\phi_2(q)=\phi_2(r)=\phi_2(u)=\phi_2(v)=\phi_2(w)=\phi_2(w)=\phi_2(a)=1$$
$$\phi_2(s)=0,\phi_2(x)=\phi_2(y)=\phi_2(n)=6.$$

\noindent Given this information, Knottwister gives us the
following representation $\alpha_2$
 (in one-line permutation form):

 $$\alpha_2:\pi_1(M) \to S_5 \to \mbox{GL}(\Bbb{Z},5)$$

\[ \begin{array}{rclrclrclrcl}
a &\mapsto& (24513)&c &\mapsto& (45132)&d &\mapsto& (35124)&e &\mapsto& (23154)\\
f &\mapsto& (21534)&g &\mapsto& (24513)&h &\mapsto& (45213)&i &\mapsto& (53214)\\
j &\mapsto& (54231)&k &\mapsto& (51432)&l &\mapsto& (45213)&n &\mapsto& (12345)\\
o &\mapsto& (41523)&p &\mapsto& (54231)&q &\mapsto& (25431)&r &\mapsto& (54123)\\
s &\mapsto& (12345)&t &\mapsto& (23154)&u &\mapsto& (54231)&v &\mapsto& (54231)\\
w &\mapsto& (54123)&x &\mapsto& (12345)&y &\mapsto& (12345).\\
\end{array} \]

\noindent For this representation, the twisted Alexander
polynomial, $\Delta_{N,\phi_2}^{\alpha_2}$, vanishes over $\Bbb{F}_{5}[t^{\pm 1}]$ and all of the fields previously mentioned for $\Delta_{N,\phi_1}^{\alpha_1}$. Hence neither twisted Alexander polynomial is monic as
claimed.\end{proof}

Again, by Theorem \ref{FK}, the pairs $(N, (0,1))$ and $(N,(1,-1))$ are not
fibered.

\subsection{Links $L_\beta$ and $L_\gamma$}
In this section, we briefly discuss the 2-component link that results from adding an arrowhead vertex to the middle node, $L_\beta$, and the one that results from adding it to the first node, $L_\gamma$. We use the theorems in \cite{EN} to conclude the following propositions.\\

\begin{proposition}\label{calcsforlbeta}
For the link $L_\beta$ the following are true:\

 1. The Alexander polynomial vanishes.\

2. The Thurston norm of the class $\phi=(p,q)$ on $L_\beta$ is
$|p+q|$.\

3. No cohomology class $\phi$ on $L_\beta$ is fibered.

\end{proposition}

\begin{proposition}\label{calcsforlgamma}
The link $L_\gamma$ has the following properties.\

1. The Alexander polynomial vanishes.\

2. The Thurston norm for a class $\phi=(p,q)$ on this link is
$7|p|+|6p+q|$.\

3. No class $\phi$ on this link is fibered. \end{proposition}

We can use similar techniques to find the fundamental groups of these links. We have discussed the three ``building blocks" of $L_\gamma$ already. For the link $L_\beta$, notice that the middle node gives the splice diagram of a 4-component necklace. The following propositions give the fundamental groups of the exteriors of these two links.

\begin{proposition}
The fundamental group of $L_\beta$ is the following:
$$\pi_1(L_\beta)=\langle x, y, a, b, s, r, t, n | aba=bab, xyx=yxy, $$
$$ nr=rn, nt=tn, ns=sn, x=n, s=x^{-1}yx^2yx^{-3}, a=rst,
n=a^{-1}ba^2ba^{-3}\rangle.$$
\end{proposition}

\begin{proposition}
The fundamental group of $L_\gamma$ is the
following group:
$$\pi_1(L_\gamma)=\langle a, b, n, s, t, c, d, e, f, g, h, i, j, k, o, l, p, q, r, u, v, w |$$
$$ gd=cg, ve=dv, cf=ec, pg=fp, vh=gv, wi=hw, xj=ix, ek=je,
rc=kr, $$
$$eo=le, rp=or, gq=pg, vr=qv, cu=rc, pv=up, hw=vh, ix=wi,
jl=xj, $$
$$ aba=bab, ns=sn, nt=tn, a=st, n=a^{-1}ba^2ba^{-3}, e=n,
s=cpvwxergve^{-3}\rangle.$$
\end{proposition}

\section{A ``Secondary" Polynomial,
$\widetilde{\Delta}_1^\alpha(t)$}

Since the ordinary Alexander
polynomial is $0$ for $L_\beta$ and $L_\gamma$, we may not use
Theorem \ref{mcm} to get a useful bound for the Thurston norm.
From now on, we will only be concerned with the single-variable
version of the twisted Alexander polynomial for simplicity. Also,
since $\mathbb{F}[t^{\pm1}]$ is a principal ideal domain, we replace $\mathbb{Z}[t^{\pm1}]$ by $\mathbb{F}[t^{\pm1}]$ in
the definition of the Alexander module where
$\mathbb{F}=\mathbb{F}_p$ is a field. As a result, we have the following isomorphism:
$$ H_1 (N, \mathbb{F}^k[t^{\pm 1}])\cong
\mathbb{F}[t^{\pm1}]^r\oplus \displaystyle\bigoplus_{j=1}^m
\mathbb{F}[t^{\pm 1}]/(p_j(t))$$ for $p_1(t),...,p_m(t) \in
\mathbb{F}[t^{\pm 1}].$ The type of polynomials we will examine
are defined by:
$$\widetilde{\Delta}_{N,\phi}^\alpha := \displaystyle \prod
_{j=1}^ m p_j(t)$$ regardless of the rank $r$. Not much is
known about these polynomials.

 S. Friedl and T. Kim
have proved the following theorem that relates these
polynomials to the Thurston norm in \cite{FK}.

\begin{theorem} \label{tildapolys}\textnormal{(Friedl-Kim, \cite{FK})} Let $L=L_1\cup L_2 \cup .... \cup L_m$ be a
link with $m$ components. Denote its meridian by
$\mu_1,...,\mu_m$. Let $\phi \in H^1(X(L);\mathbb{Z})$, be
primitive and dual to a meridian $\mu_i$,when $X(L)$ denotes
the exterior of $L$. Hence $\phi(\mu_i)=1$ for some $i$ and
$\phi(\mu_j)=0$ for $j\neq i$. Then
$$\| \phi \|_T \geq
\frac{1}{k}deg(\widetilde{\Delta}_1^\alpha(t))-1.$$ Here, $k$ is
the size of the representation $\alpha$.

\end{theorem}

Theorem \ref{tildapolys} will help us improve the bound of the
Thurston norm for the class $(0,1)$ for both $L_\beta$ and
$L_\gamma$. Recall from Section 2.2 that for $L_\beta$, the
Thurston norm of a general cohomology class $(p,q)$ is $|p+q|$.
So for this link, $\|(0,1)\|_T=1$. In this case, Knottwister
computes the $\widetilde{\Delta}_1^\alpha(t)$ to be $1-t+t^2$
over $\mathbb{F}_{13}$ when $\alpha$ is trivial (so $k=1$).
Therefore, for the pair $(L_\beta,(0,1))$ we get
$$\|(0,1)\|_T \geq 2-1=1$$ which is a sharp bound.

Now, we consider the same cohomology classes on $L_\gamma$. We
know from our calculations in section 2.2 that for this link,
$\|\phi\|_T=\|(p,q)\|_T= 7|p|+|6p+q|$. So for this link
$\|(0,1)\|_T=1$. Knottwister yields the
$\widetilde{\Delta}_1^\alpha(t)= 1-t+t^2$ over
$\mathbb{F}_{13}$ again, when $\alpha$ is trivial, which is
again a sharp bound.

\bibliographystyle{amsplain}

\end{document}